\documentclass[leqno,10pt,a4paper]{amsart}
\usepackage[full]{textcomp}
\usepackage{newpxtext}
\usepackage{cabin} 
\usepackage[varqu,varl]{inconsolata} 
\usepackage[bigdelims,vvarbb]{newpxmath} 
\usepackage[cal=cm,bb=esstix,bbscaled=1.05]{mathalfa} 
\usepackage[margin=2.4cm]{geometry}
\usepackage{savesym}

\usepackage{mathabx}
\usepackage[usenames,dvipsnames]{color}
\usepackage{hyperref}
\hypersetup{
 colorlinks,
 linkcolor={red!50!black},
 citecolor={blue!50!black},
 urlcolor={blue!80!black}
}
\usepackage{tikz}
\usepackage{graphicx}
\usepackage[all,cmtip]{xy}
\usepackage{mleftright}
\usepackage{booktabs}
\usepackage{cite}
\usepackage{enumitem}

\setlist[enumerate]{labelsep=*, leftmargin=1.5pc}
\setlist[enumerate]{label=\normalfont(\roman*), ref=\roman*}
\newtheorem{thm}{Theorem}[section]
\newtheorem{lemma}[thm]{Lemma}
\newtheorem{cor}[thm]{Corollary}

\newtheorem{conjecture}[thm]{Conjecture}
\theoremstyle{definition}
\newtheorem{example}[thm]{Example}

\newtheorem{definition}[thm]{Definition}

\numberwithin{equation}{section}
\newcommand{\Z}{\mathbb{Z}}
\newcommand{\Q}{\mathbb{Q}}
\newcommand{\R}{\mathbb{R}}
\newcommand{\C}{\mathbb{C}}

\newcommand{\mO}{\mathcal{O}}

\newcommand{\ph}{\varphi}

\newcommand{\pr}{\mathbb{P}}

\newcommand{\on}{\operatorname}

\newcommand{\wt}{\widetilde}

\newcommand{\mfk}{\mathfrak}

\graphicspath{{images/}}
\begin{document}
\author[B.\,Wormleighton]{Ben~Wormleighton}
\address{Department of Mathematics\\University of California at Berkeley\\Berkeley, CA\\94720\\USA}
\email{b.wormleighton@berkeley.edu}
\title{ECH capacities, Ehrhart theory, and toric varieties}
\maketitle

\begin{abstract}
ECH capacities were developed by Hutchings to study embedding problems for symplectic $4$-manifolds with boundary. They have found especial success in the case of certain toric symplectic manifolds where many of the computations resemble calculations found in cohomology of $\Q$-line bundles on toric varieties, or in lattice point counts for rational polytopes. We formalise this observation in the case of convex toric lattice domains $X_\Omega$ by constructing a natural polarised toric variety $(Y_{\Sigma(\Omega)},D_\Omega)$ containing the all the information of the ECH capacities of $X_\Omega$ in purely algebro-geometric terms. Applying the Ehrhart theory of the polytopes involved in this construction gives some new results in the combinatorialisation and asymptotics of ECH capacities for convex toric domains.
\end{abstract}

\section{Introduction}

Symplectic capacities measure obstructions to embedding one symplectic manifold into another. Perhaps the simplest such obstruction is the volume; a symplectic manifold $(X_1,\omega_1)$ can be embedded in another symplectic manifold $(X_2,\omega_2)$ only if $\on{vol}(X_1,\omega_1)\leq\on{vol}(X_2,\omega_2)$. A more sophisticated obstruction is the Gromov width: the supremum of the radii of balls that can symplectically embed into the given symplectic manifold. As Gromov's nonsqueezing theorem \cite{grom} illustrates, this is a nontrivial and interesting invariant even for simple submanifolds of $\R^n$.
\\

There are many different capacities in past and current usage - see \cite{chls} and the numerous references therein for an overview - that were invented in order to answer more sophisticated embedding questions about symplectic $4$-manifolds. In this paper we will focus on \textit{Embedded Contact Homology} or \textit{ECH capacities}, which were developed by Hutchings in \cite{h11} and have since been studied by many authors in, for example, \cite{ccfhr}, \cite{chr}, \cite{cg}, \cite{cs}. To an exact symplectic $4$-manifold $X$ with contact-type boundary they associate an increasing sequence of real numbers $c_k(X)$ for $k\in\Z_{\geq0}$. One of their early successes was studying embeddings of ellipsoids where the ellipsoid with symplectic radii $a,b$
$$E(a,b):=\{(x,y)\in\C^2:|x^2|/\pi a+|y^2|/\pi b\leq 1\}$$
embeds into $E(c,d)$ iff $c_k(E(a,b))\leq c_k(E(c,d))$ for all $k$. Moreover, $c_k(E(a,b))$ was computed to be the $k$th largest number of the form $am+bn$ for $m,n\in\Z_{\geq0}$.
\\

A particular type of symplectic manifold that ECH capacities provide an attractive means of studying is toric domains. Consider the moment map
$$\mu:\C^2\to\R^2$$
for the $2$-torus action on $\C^2$. Given a region $\Omega\subset\R^2$, $X_\Omega:=\mu^{-1}(\Omega)$ is a toric symplectic $4$-manifold potentially with boundary. If the domain $\Omega$ is a certain kind of convex polygon with two edges lying on the coordinate axes, $X_\Omega$ is called a \textit{convex toric domain}. We omit mention of the symplectic form since we will always take the induced form from $\C^2$. Such symplectic manifolds are exact with contact-type boundary. The work of Cristofano-Gardiner--Choi \cite{cg} provides a somewhat combinatorial formula for the ECH capacities of such spaces in terms of lattice paths and lattice point counts.
\\

Define the \textit{cap function} of a symplectic $4$-manifold $X$ with contact-type boundary to be
\begin{align*}
\on{cap}_X(r)&:=\#\{k\in\Z_{\geq0}:c_k(X)\leq r\} \\
&=1+\on{max}\{k\in\Z_{\geq0}:c_k(X)\leq r\}
\end{align*}
for $r\in\Z_{\geq0}$. In certain situations - such as ellipsoids with integral symplectic radii - the cap function recovers all of the ECH capacities. The main results of this paper apply to the capacities and cap functions of convex toric domains where $\Omega$ is in addition a lattice polygon.
\\

For a rational polygon $\Omega\subset\R^2$ we consider the inner normal fan $\Sigma(\Omega)$, which is the complete fan whose rays are the (primitive) inward-pointing normal vectors to the edges of $\Omega$. This defines a toric variety $Y_{\Sigma(\Omega)}$. We will later define a divisor $D_\Omega$ on $Y_{\Sigma(\Omega)}$ called the \textit{balance divisor}. The key property of this divisor is that its associated polytope is equal to $\Omega$.
\\

Recall that the function counting lattice points in dilates of a lattice polytope $P\subset\R^n$ is given by a polynomial $\on{ehr}_P$, called the Ehrhart polynomial of $P$, such that $\#(nP)\cap\Z^n=\on{ehr}_P(n)$ for $n\in\Z_{\geq0}$. Similarly, recall that the function counting global sections in integer multiples of a Cartier divisor $D$ on a variety $X$ is eventually given by a polynomial $\on{hilb}_{(X,D)}$, called the Hilbert polynomial of $(X,D)$, such that $h^0(X,nD):=\on{dim} H^0(X,nD)=\on{hilb}_{(X,D)}(n)$ for all sufficiently large $n\in\Z_{\geq0}$. 
\\

When $P$ is a rational polytope (or $D$ is a $\Q$-Cartier divisor), the Ehrhart function (resp. the Hilbert function) is given (resp. eventually given) by a \textit{quasipolynomial}: there exists a number $\pi\in\Z_{\geq1}$ and polynomials $L_0,\dots,L_{\pi-1}$ such that
$$L_P(n)=L_i(n)\text{ when $n\equiv i\on{mod}{\pi}$}$$
Alternatively, one can think of such a function as a polynomial with coefficients that are periodic functions.
\\

With the above setup, the main results of this paper are as follows. Let $X_\Omega$ be a convex toric domain. Choi--Cristofaro-Gardiner--Frenkel--Hutchings--Ramos \cite{ccfhr} and Cristofaro-Gardiner \cite{cg} associate a sequence of numbers to $\Omega$ called the \text{weight sequence} $w(\Omega)$. We will later define in Definition \ref{def:tight} a class of convex toric domains that are \textit{tightly constrained}. We conjecture (Conjecture \ref{conj:wt}) that this is equivalent to the gcd of the numbers in the weight sequence being equal to $1$; it is shown below to hold when one of the weights is equal to $1$ in Lemma \ref{lem:wt}.

\begin{thm} (Corollary \ref{cor:main}) \label{thm:1} Suppose $X_\Omega$ is a tightly constrained convex toric lattice domain. Then there exists some $r_0\in\Z_{\geq0}$ such that $\on{cap}_{X_\Omega}(r)$ is given by a quasipolynomial of known period $\lambda$ for all $r\geq r_0$. More precisely, for $r=0,\dots,\lambda-1$
$$\on{cap}_{X_\Omega}(r+\lambda x)=\on{ehr}_\Omega(x)+rx+\gamma_r$$
for some constant $\gamma_r\in\Z$ dependent only on $r$, whenever $r+\lambda x\geq r_0$.
\end{thm}

\begin{thm} (Corollary \ref{cor:main}) \label{thm:2} Suppose $X_\Omega$ is a tightly constrained convex toric domain. Then, if $D_\Omega$ is the balance divisor on $Y_{\Sigma(\Omega)}$ associated to $\Omega$, there is $r_0\in\Z_{\geq0}$ such that for any $r=0,\dots,\lambda-1$ and $x\in\Z_{\geq0}$ with $r+\lambda x\geq r_0$
$$\on{cap}_{X_\Omega}(r+\lambda x)=h^0(Y_{\Sigma(\Omega)},xD_\Omega)+rx+\gamma_r$$
for some constant $\gamma_r\in\Z$ dependent only on $r$.
\end{thm}

Moreover, when at least one weight is equal to $1$ one can choose $r_0=0$ and $\gamma_r=\on{cap}_{X_\Omega}(r)-1$ in both theorems.  Letting the residue $r$ mod $\lambda$ be zero in the formulae above gives the following corollary.

\begin{cor} \label{cor:zero} If $X_\Omega$ is a tightly constrained convex toric lattice domain, then for sufficiently large $x\in\Z_{\geq0}$
$$\on{cap}_{X_\Omega}(\lambda x)=\on{ehr}_\Omega(x)+\gamma_0=h^0(xD_\Omega,Y_{\Sigma(\Omega)})+\gamma_0$$
for some $\gamma_0\in\Z$.
\end{cor}

When at least one weight is equal to $1$ here we have $\gamma_0=0$. The explicit description of the linear coefficients above give precise examples of sub-leading asymptotics for ECH capacities as studied in \cite{cs}; for example, Prop. 16 there is an interesting comparison.
\\

We conjecture that the following strengthening of the prior results holds.

\begin{conjecture} \label{conj:2} Suppose that $X_\Omega$ is a tightly constrained toric domain. Then:
\begin{itemize}
\item there exist convex lattice domains $\Omega_0,\dots,\Omega_{\lambda-1}$ such that, for any $r=0,\dots,\lambda-1$ and any sufficiently large $x\in\Z_{\geq0}$,
$$\on{cap}_{X_\Omega}(r+\lambda x)=|(\Omega_r+x\Omega)\cap\Z^2|$$
\item there exist divisors $D_0,\dots,D_{\lambda-1}$ on $Y_{\Sigma(\Omega)}$ such that, for any $r=0,\dots,\lambda-1$ and any sufficiently large $x\in\Z_{\geq0}$,
$$\on{cap}_{X_\Omega}(r+\lambda x)=h^0(Y_{\Sigma(\Omega)},D_r+xD_\Omega)$$
\end{itemize}
Moreover, we conjecture that $\Omega_0=\{0\}$ so that $\gamma_0=0$, and that all of these claims actually hold for all $x\in\Z_{\geq0}$, not just for all sufficiently large $x$.
\end{conjecture}

These conjectures state that $\on{cap}_{X_\Omega}$ is (eventually) given by a `mixed Ehrhart quasipolynomial' or a `mixed Hilbert quasipolynomial' as studied in \cite{hjst}. Cristofaro-Gardiner--Kleinman in \cite{ck} have previously approached symplectic embeddings problems for ellipsoids via Ehrhart theory, and one can view some aspects of the present paper as pursuing a related philosophy for convex toric lattice domains.
\\

As the results above might suggest, there is a purely algebro-geometric framework that we will establish in which one can recast ECH capacities with no conditions except rationality on the weights of $\Omega$. This also works for a different class of toric domains called \textit{free convex toric domains} that are defined in \S\ref{sec:free}.

\begin{thm} \label{thm:main3} (Theorem \ref{thm:transl} + Theorem \ref{thm:reform} + Theorem \ref{thm:free}) Let $\Omega$ be any rational convex lattice domain or a free rational convex toric domain. Then
\begin{align*}
&c_k(X_\Omega)=\on{min}\{D\cdot D_\Omega:h^0(Y_{\Sigma(\Omega)},D)\geq k+1\} \\
&\on{cap}_{X_\Omega}(r)=\on{max}\{h^0(Y_{\Sigma(\Omega)},D):D\cdot D_\Omega\leq r\}
\end{align*}
where both extrema range over all nef $\Q$- or $\R$-divisors on $Y_{\Sigma(\Omega)}$.
\end{thm}

For the special case evaluating the cap function at $r\ell_\Omega(\partial\Omega)=:\lambda r$, we have
$$\on{cap}_{X_\Omega}(\lambda r)=\on{max}\{h^0(D):(D-rD_\Omega)\cdot D_\Omega\leq 0\}$$

We can also state some of the results of this paper in purely combinatorial terms. We will later describe a pseudonorm $\ell_\Omega$ dependent on $\Omega$ called the $\Omega$\textit{-length}, which is central to the combinatorialisation of ECH capacities. For a polygon $\Lambda$, define its $\Omega$\textit{-perimeter} $\ell_\Omega(\partial\Lambda)$ to be the sum of the $\Omega$-lengths of the line segments composing its boundary $\partial\Lambda$.

\begin{thm}(Corollary \ref{cor:poly}) Suppose $\Omega$ is a tightly constrained convex lattice domain with lower bound\footnote{For example, these assumptions are met if one of the weights of $\Omega$ is equal to $1$, and we conjecture that they are met whenever the gcd of the weights is $1$.} $r_0=0$ and let $\lambda=\ell_\Omega(\partial\Omega)$. Then $r\Omega$ contains the most lattice points of any convex lattice domain of $\Omega$-perimeter at most $r\lambda$ for all $r\in\Z_{\geq0}$.
\end{thm}

All of these results are numerical in nature and so it is natural to wonder if there is some higher structure behind them. In particular, both sides of the equality in Theorem \ref{thm:2} are defined to be dimensions of vector spaces - of filtered embedded contact homology on the left, and of cohomology of divisors on $Y_{\Sigma(\Omega)}$ on the right - and so it would be interesting to explore whether there is a correspondence on the level of vector spaces, potentially accessed by mirror symmetry.

\subsection{Acknowledgements}

I am very grateful to Michael Hutchings for introducing me to ECH capacities and for many fruitful discussions as the content of this paper developed. I would also like to thank Vivek Shende for useful conversations at various points of this story, and Dan Cristofaro-Gardiner for comments on a draft of this paper.

\section{ECH capacities}

ECH is formally defined in terms of contact geometry. It is constructed explicitly in \cite{h11} however there is a combinatorial rephrasing of ECH in the case of toric domains that is most applicable to the situation at hand, which is how we will primarily present it here. This material comes from \cite{h11}, \cite{ccfhr}, and \cite{cg}.

\subsection{Combinatorial definitions} \label{sec:comb}

Suppose $\Omega\subset\R^2$ is any polygon. Define the $\Omega$\textit{-length} of a vector $v$ to be
$$\ell_\Omega(v):=v\times p_v$$
where $p_v$ is a boundary point of $\Omega$ such that $\Omega$ is contained in the right halfplane bounded by the line spanned by $v$ translated to contain $p_v$. Here $\times$ means the cross product $u\times v=\on{det}(u\mid v)$. Define the $\Omega$-length of a piecewise linear path $\Lambda$ to be
$$\ell_\Omega(\Lambda)=\sum\ell_\Omega(v_i)$$
where the sum ranges over the edge vectors $v_i$ of $\Lambda$. Notice that, from a local calculation, one has
$$\ell_\Omega(\partial\Omega)=2\on{Vol}(\Omega)$$

\begin{definition} A \textit{convex domain} is a convex region $\Omega\subset\R^2$ whose boundary consists of
\begin{itemize}
\item a line segment between the origin and a point $(a,0)$ on the positive horizontal axis
\item a line segment between the origin and a point $(0,b)$ on the positive vertical axis
\item the graph of a convex piecewise linear function $f:[0,a]\to[0,b]$
\end{itemize}
We say that a convex domain is a \textit{convex lattice domain} if the points $(a,0)$ and $(0,b)$ are lattice points and if the function $f$ is piecewise linear such that each vertex is a lattice point. In other words, a convex lattice domain is a convex domain that is also a lattice polygon. \textit{Convex rational domains} are defined similarly.
\end{definition}

We call the corresponding symplectic manifold $X_\Omega=\mu^{-1}(\Omega)$ a \textit{convex toric domain} if $\Omega$ is a convex domain, or a \textit{convex toric lattice domain} if $\Omega$ is a convex lattice domain. One can also repeat these definitions with convex replaced by \textit{concave}.
\\

Following \cite{cg} - which built on \cite{ccfhr} and \cite{mcd} - the \textit{weight sequence} associated to a convex lattice domain $\Omega$ is a sequence $w(\Omega)$ of numbers defined as follows. Let $\Delta_a$ be the convex hull of the points $(0,0),(a,0),(0,a)$. Let $c$ be the smallest number such that $\Omega\subset\Delta_c$. Equivalently, $c$ is the radius of the smallest ball in $\C^2$ containing $X_\Omega$. The two components of the complement $\Delta_c\setminus\Omega$ are affine equivalent to two concave domains $\Omega_2$ and $\Omega_3$. There is a recursive definition weight sequences for concave domains as follows. Consider the concave domain $\Omega_2$. Let $b_1$ be the largest real number such that $\Delta_{b_1}\subset \Omega_2$. The complement of $\Delta_{b_1}$ in $\Omega_2$ consists of two (possibly empty) concave domains and so one can recurse to obtain a multiset of numbers $w(\Omega_2):=\{b_1,b_2,\dots\}$. We define
$$w(\Omega):=(c;w(\Omega_2);w(\Omega_3))$$

\begin{example} Let $\Omega=\on{Conv}((0,0),(0,2a),(a,a),(a,0))$ for some $a\in\Z_{>0}$.
\begin{figure}[h]
\caption{Example of weight sequence}
\begin{center}
\begin{tikzpicture}[scale=0.8]
\node (o) at (0,0){$\bullet$};
\node (a) at (0,4){$\bullet$};
\node (b) at (2,2){$\bullet$};
\node (c) at (2,0){$\bullet$};

\draw (o.center) to (a.center) to (b.center) to (c.center) to (o.center);

\node (o) at (4,0){$\bullet$};
\node (a) at (4,4){$\bullet$};
\node (b) at (6,2){$\bullet$};
\node (c) at (6,0){$\bullet$};
\node (d) at (8,0){};

\draw (o.center) to (a.center);
\draw (b.center) to (c.center) to (o.center);
\draw[dashed] (a.center) to (d.center) to (c.center);

\node (b) at (10,2){$\bullet$};
\node (c) at (10,0){$\bullet$};
\node (d) at (12,0){$\bullet$};

\draw (b.center) to (d.center) to (c.center) to (b.center);
\end{tikzpicture}
\end{center}
\end{figure}
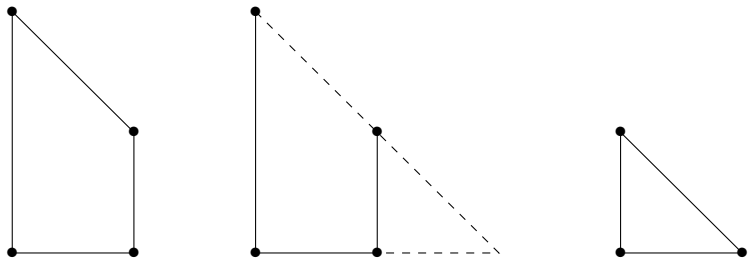
Here $c=2a$, leaving a single concave region $\Omega_3$ illustrated in the third figure, which is affine equivalent to $\Delta_a$. The weight sequence for $\Omega$ is hence $(2a;\emptyset;a)$.
\end{example}

\subsection{ECH capacities}

Using the constructions above, we define ECH capacities combinatorially.

\begin{definition} \label{def:path} A \textit{convex lattice path} is a piecewise linear path starting on the positive vertical axis and ending on the positive horizontal axis such that its vertices are lattice points.
\end{definition}

After adding the pieces along the coordinate axes, convex lattice paths are exactly boundaries of convex lattice domains. For a polygon $\Lambda$ we denote by $L_\Lambda$ the number of lattice points enclosed by $\Lambda$, including on those its boundary. We state a result of Cristofaro-Gardiner but using the perspective of convex lattice domains instead of convex lattice paths.

\begin{thm}[\cite{cg}, Cor. 8.5] \label{thm:opt} Let $\Omega$ be a convex domain. Then
$$c_k(X_\Omega)=\on{min}\{\ell_\Omega(\partial\Lambda):L_\Lambda=k+1\}$$
where the minimum is taken over convex lattice domains $\Lambda$.
\end{thm}

\begin{cor} If $\Omega$ is a convex domain, then
$$\on{cap}_{X_\Omega}(r)=\on{max}\{L_\Lambda:\ell_\Omega(\partial\Lambda)\leq r\}$$
where the maximum ranges over convex lattice domains $\Lambda$.
\end{cor}

\begin{proof} After including the zeroth capacity, one has
\begin{align*}
\on{cap}_{X_\Omega}(r)&=\#\{k:\text{$\exists\Lambda$ with $\ell_\Omega(\Lambda)\leq r$ and $L_\Lambda=k+1$}\} \\
&=1+\on{max}\{k:\text{$\exists\Lambda$ with $\ell_\Omega(\Lambda)\leq r$ and $L_\Lambda=k+1$}\} \\
&=\on{max}\{L_\Lambda:\ell_\Omega(\Lambda)\leq r\}
\end{align*}
as required. 
\end{proof}

\subsection{ECH capacities and weight sequences}

The weight sequence $w(\Omega)$ contains all the information required to compute $c_k(X_\Omega)$.

\begin{lemma} \label{lem:sharp} Suppose $w(\Omega)=(c;a_1,\dots,a_s;b_1,\dots,b_t)$. Then
$$c_k(X_\Omega)=\on{min}\{c_{k+k_2+k_3}(B(c))-c_{k_2}(\amalg_{i=1}^sB(a_i))-c_{k_3}(\amalg_{j=1}^tB(b_j)):k_2,k_3\in\Z_{\geq0}\}$$
\end{lemma}

This follows from \cite{cg} Corollary A.5 combined with \cite{ccfhr} Theorem 1.4.

\subsection{Key properties of ECH capacities} \label{sec:prop}

ECH capacities have the following properties recorded in \cite{ccfhr}, which we will use throughout the paper:
\begin{itemize}
\item \textbf{Monotonicity:} If $(X,\omega)$ embeds into $(X',\omega')$ then $c_k(X,\omega)\leq c_k(X',\omega')$ for all $k$
\item \textbf{Disjoint union:} If $(X,\omega)=\amalg_{i=1}^n(X_i,\omega_i)$ then
$$c_k(X,\omega)=\underset{\sum k_i=k}{\on{max}}\sum_{i=1}^nc_{k_i}(X_i,\omega_i)\text{ for all $k$}$$
\item \textbf{Conformality:} For each $k$ and $\lambda\in\R^+$, $c_k(X,\lambda\omega)=\lambda c_k(X,\omega)$
\end{itemize}

\subsection{Asymptotics of ECH capacities}

Asymptotically, capacities return the volume constraint on symplectic embeddings.

\begin{thm}[\cite{chr}, Theorem 1.1] \label{thm:asy} Suppose $\Omega$ is a convex domain, then
$$\underset{k\to\infty}{\on{lim}}\frac{c_k(X)^2}{k}=4\on{Vol}(X_\Omega)=4\on{Vol}(\Omega)$$
\end{thm}

This already allows us to calculate the constant $\lambda$ in Theorems \ref{thm:1} and \ref{thm:2} assuming that it exists. Theorem \ref{thm:asy} implies that $c_k(X)$ is asymptotic to $2\sqrt{k\on{Vol}(\Omega)}$ and so the cap function is asymptotic to
$$\#\{k:2\sqrt{k\on{Vol}(\Omega)}\leq r\}=\#\{k:k\leq \frac{1}{4\on{Vol}(\Omega)}r^2\}=\frac{1}{4\on{Vol}(\Omega)}r^2+1$$
The leading term of the Ehrhart polynomial in Theorem \ref{thm:1} and of the Hilbert polynomial in Theorem \ref{thm:2} is $\on{Vol}(\Omega)n^2$ and so the constant should be
$$\lambda=2\on{Vol}(\Omega)=\ell_\Omega(\partial\Omega)$$

\subsection{Examples}

We present some suggestive examples of calculations of capacities and cap functions for some basic convex toric domains.

\begin{example} $\on{cap}_{E(a,b)}(r)=\on{ehr}_Q(r)$, the Ehrhart quasipolynomial of the rational triangle
$$Q=\on{Conv}((0,0),(1/a,0),(0,1/b))$$
This is also equal to the Hilbert function of $\mO(1)$ for the weighted projective plane $\pr(1,a,b)$.
\end{example}

\begin{example} The cap function for the polydisk $P(a,b)$ has
$$\on{cap}_{P(a,b)}(2abr)=(ar+1)(br+1)=\on{hilb}_{(\pr^1\times\pr^1,\mO(a,b))}(r)$$
\end{example}

\begin{example} Let $\Omega(a)$ be the convex hull of the points $(0,0),(0,2a),(a,a),(a,0)$. One has
$$\on{cap}_{X_{\Omega(a)}}(3ar)=h^0(X,rD)$$
where $X$ is the first Hirzebruch surface, or $\pr^2$ blown up in one point, and where $D=3C+2F$ with $C$ the $(-1)$-curve and $F$ a fibre in the $\pr^1$-bundle structure on $X$.
\end{example}

These examples all suggest a tight relationship between computations of symplectic capacities for convex toric domains and Hilbert functions of divisors on toric surfaces. Establishing and exploiting such a relationship is the subject of the remainder of this paper.

\section{Toric algebraic geometry}

We begin by reviewing some basic toric algebraic geometry. A toric variety is a partial compactification of an algebraic torus $(\C^\times)^n$. They are described combinatorially by cones, fans, and polytopes. This and much more is detailed in \cite{cls}.

\subsection{Affine toric varieties arise from cones}
Let $N\cong\Z^n$ be a lattice and let $N_\R:=N\otimes_\Z\R$ be the associated real vector space. A \textit{cone} $\sigma$ in $N_\R$ is a  subset of the form
$$\on{Cone}(S):=\{\sum_{v\in S}\lambda_vv:\lambda_v\geq0,\text{all but finitely many $\lambda_v$ are zero}\}$$
Let $M=N^\vee:=\on{Hom}_\Z(N,\Z)$ be the dual lattice to $N$, and $M_\R=M\otimes_\Z\R$ the dual vector space to $N_\R$. Define the \textit{dual cone} to a cone $\sigma\subset N_\R$ to be
$$\sigma^\vee:=\{v\in M_\R:\langle u,v\rangle\geq0\text{ for all $u\in\sigma$}\}$$
where $\langle\cdot,\cdot\rangle$ is the dual pairing $N_\R\times M_\R\to\R$. Suppose now that $\sigma$ is a \textit{rational polyhedral} cone: that there is a finite set of lattice points $S\subset N$ such that $\sigma=\on{Cone}(S)$. Such a cone $\sigma$ gives an affine toric variety $U_\sigma$ as follows.
\begin{itemize}
\item \textbf{Input:} $\sigma$, a rational polyhedral cone
\item Dualise to $\sigma^\vee$
\item Take lattice points $\sigma^\vee\cap M$ to obtain a semigroup
\item Take the semigroup algebra $\C[\sigma^\vee\cap M]$; this is a finitely generated $\C$-algebra
\item \textbf{Output:} $U_\sigma:=\on{Spec}{\C[\sigma^\vee\cap M]}$.
\end{itemize}
Notice that the dense open torus arises from $\C[\sigma^\vee\cap M]\subset\C[M]\cong\C[\Z^n]$, which is the ring of Laurent polynomials, or the ring of functions for the torus $(\C^\times)^n$. The cone $\sigma$ (or rather $\sigma^\vee$) is describing which functions on the torus extend to global functions on $U_\sigma$, which is equivalent to describing the variety. One can describe the torus inside $U_\sigma$ intrinsically as
$$T_N:=N\otimes_\Z\C^\times$$
In this presentation, a vector $m\in M$ gives a function $\chi^m:T_N\to\C$ via
$$\chi^m(n\otimes t)=t^{\langle m,n\rangle}$$

\begin{example} Take $N=\Z^2$ and let $\sigma=\on{Cone}(e_1,e_2)$. The dual cone is $\sigma^\vee=\on{Cone}(e^1,e^2)$ giving
$$\sigma^\vee\cap M=\Z_{\geq0}^2\text{ and }\C[\sigma^\vee\cap M]\cong\C[x,y]$$
Hence $U_\sigma\cong\C^2$. In this case, $\sigma^\vee$ prescribes that the only Laurent polynomials extending to all of $U_\sigma$ are the polynomials.
\end{example}

\subsection{Toric varieties arise from fans} To construct non-affine (in particular, compact) toric varieties we glue together affine toric varieties in an torus-equivariant way. The combinatorial avatar of this process is collecting cones together in a \textit{fan}. To start with, a \textit{face} of a cone $\sigma$ is a subset of $\sigma$ of the form $\sigma\cap(\langle m,\cdot\rangle=0)$ for some $m\in\sigma^\vee$. The cones forming the boundary of $\sigma$ are examples of faces, as is the vertex of the cone (the origin). A \textit{fan} in $N_\R$ is a collection of cones $\Sigma=\{\sigma\}$ such that
\begin{itemize}
\item if $\tau\subset\sigma$ is a face, then $\tau\in\Sigma$
\item for any two cones $\sigma_1,\sigma_2\in\Sigma$, $\sigma_1\cap\sigma_2$ is a face of each
\end{itemize}
A fan $\Sigma$ produces a toric variety $Y_\Sigma$ via gluing two affine pieces $U_{\sigma_1},U_{\sigma_2}$ according to the (potentially zero-dimensional) face they have in common.

\begin{example} Take $N=\Z^2$ and $\Sigma$ to be the fan containing the cones $\sigma_1=\on{Cone}(e_1,e_2),\sigma_2=\on{Cone}(e_1,-e_1-e_2),\sigma_3=\on{Cone}(e_2,-e_1-e_2)$ and their faces. The two-dimensional cones give three copies of $\C^2$ and the gluing prescribed by the faces makes this into $\pr^2$. For example, $\sigma_1$ and $\sigma_3$ share the face $\on{Cone}(e_2)$ that corresponds to the toric variety $\C^\times\times\C$. Gluing $\C^2$ to $\C^2$ along $\C^\times\times\C$ is familiar from the gluing construction of projective space.
\end{example}

\subsection{Compact toric varieties arise from polytopes} Suppose $P\subset N_\R$ is a lattice polytope. One can produce a fan $\Sigma_P$ from $P$ via
$$\Sigma_P:=\{\on{Cone}(S):S\subset\on{Vert}(P)\text{ such that all $u\in S$ share a face}\}$$
This is called the \textit{face fan} of $P$ and defines a toric variety $Y_P:=Y_{\Sigma_P}$ that turns out to be compact.
\\

A polytope $Q\subset M_\R$ also defines a toric variety $V_Q$. Let $L_Q=\#Q\cap M$ and define a map $\phi_Q:T_N\to\pr^{L_Q-1}$ by $x\mapsto(\chi^m(x))_{m\in Q\cap M}$. The toric variety $V_Q$ is defined to be the closure of the image of $\phi_Q$ in $\pr^{L_Q-1}$. If we define the dual polytope
$$P^\vee:=\{v\in M_\R:\langle u,v\rangle\geq-1\}$$
then the toric variety $Y_P$ is also described abstractly as the variety $V_{kP^\vee}$ for large enough $k$, from which it is readily apparent that it is compact.

\begin{example}
A polytope for $\pr^2$ is the triangle with vertices $e_1,e_2,-e_1-e_2$. The dual polytope is the triangle with vertices $2e^1-e_2,-e_1+2e_2,-e^1-e^2$. This has $10$ lattice points and describes the third Veronese (or anticanonical) embedding of $\pr^2$ in $\pr^9$.
\end{example}

In the $V_Q$ presentation, one can interpret $Q$ as the moment polytope for the compact torus action on $V_Q$ by composing the map $\phi_Q$ with the moment map on $\pr^{L_Q-1}$.
\\

This toric variety $V_Q$ as an abstract variety is equivariantly isomorphic to the variety $X_{\Sigma(Q)}$ arising from the inner normal fan of $Q$.

\subsection{Polytopes arise from divisors}

A (Weil) divisor on a normal variety is a formal $\Z$-linear combination of codimension one subvarieties. Divisors on a variety $X$ up to an equivalence relation called rational equivalence form a group called the \textit{class group} of $X$. For a toric variety $X$ containing dense open torus $T$, the class group is generated by the components of the toric boundary $X\setminus T$. If $X=Y_\Sigma$ is given by a fan, these boundary components correspond to the rays of $\Sigma$. The set of rays is commonly denoted $\Sigma(1)$. Thus, every divisor on $Y_\Sigma$ is rationally equivalent to one of the form
$$\sum_{\rho\in\Sigma(1)}a_\rho D_\rho$$
One can associate a polytope $P(D)$ to a divisor of this form as follows. Let $u_\rho$ be the primitive lattice point lying on the ray $\rho$. Then set
$$P(D):=\{v\in M_\R:\langle u_\rho,v\rangle\geq-a_\rho\text{ for all $\rho\in\Sigma(1)$}\}$$
The hyperplanes defining the facets of $P(D)$ are given by $\langle u_\rho,\cdot\rangle=-a_\rho$ and so this construction of $P(D)$ taking in the data $(u_\rho,a_\rho)_{\rho\in\Sigma(1)}$ is often referred to as a `facet presentation' for $P(D)$. Denote by $\mO(D)$ the line bundle associated to a (Cartier) divisor $D$.

\begin{lemma}[\cite{cls}, Proposition 4.3.3] Let $D=\sum_\rho a_\rho D_\rho$. A basis of $H^0(\mO(D))$ is in bijection with lattice points of $P(D)$. That is,
$$\#P(D)\cap M=L_{P(D)}=h^0(\mO(D))$$
\end{lemma}

Notice that there can be multiple facet presentations corresponding to the same divisor if some of the hyperplanes give redundant inequalities.

\subsection{Divisors arise from support functions}

Fix a fan $\Sigma$. The \textit{support} $|\Sigma|$ of $\Sigma$ is the union of the cones it contains. A \textit{support function} on $\Sigma$ is a function $\ph:|\Sigma|\to\R$ such that $\ph|_\sigma$ is linear for each $\sigma\in\Sigma$. An \textit{integral} support function is a support function such that $\ph(|\Sigma|\cap N)\subset\Z$. An integral support function $\ph$ produces a (Cartier) divisor $D$ via
$$D=-\sum_{\rho\in\Sigma(1)}\ph(u_\rho)D_\rho$$
and this process is actually reversible (so long as $D$ is Cartier).

\section{Reformulating capacities in toric algebraic geometry}

\subsection{$\Omega$-stretching}

Consider a convex domain $\Omega\subset\R^2$. For a polygon $\Lambda$ define $S_\Omega\Lambda$ to be the polygon with edges parallel to the edges of $\Omega$ by placing an edge of slope $v_i$ at the point or points at which $v_i$ is tangent to $\Lambda$, using corners if necessary. For example,

\begin{figure}[h]
\caption{Example of $\Omega$-stretching}
\begin{center}
\begin{tikzpicture}
\node (o1) at (0,0){};
\node (o2) at (2,0){};
\node (o3) at (0,2){};
\node (o4) at (2,2){};

\node (o) at (1,1){$\Omega$};

\draw (o1.center) to (o2.center) to (o4.center) to (o3.center) to (o1.center);

\node (a1) at (4,0){};
\node (a2) at (6,1){};
\node (a5) at (6,2){};
\node (a3) at (5,2){};
\node (a4) at (3,1/2){};

\node (a) at (5,1.25){$\Lambda$};

\draw (a1.center) to (a2.center) to (a5.center) to (a3.center) to (a4.center) to (a1.center);

\node (a1) at (8,0){};
\node (a2) at (10,1){};
\node (a5) at (10,2){};
\node (a3) at (9,2){};
\node (a4) at (7,1/2){};

\node (b1) at (7,0){};
\node (b2) at (10,0){};
\node (b3) at (10,2){};
\node (b4) at (7,2){};

\node (b) at (7.5,1.5){$S_\Omega\Lambda$};

\draw[dashed] (a1.center) to (a2.center) to (a5.center) to (a3.center) to (a4.center) to (a1.center);
\draw (b1.center) to (b2.center) to (b3.center) to (b4.center) to (b1.center);
\end{tikzpicture}
\end{center}
\end{figure}
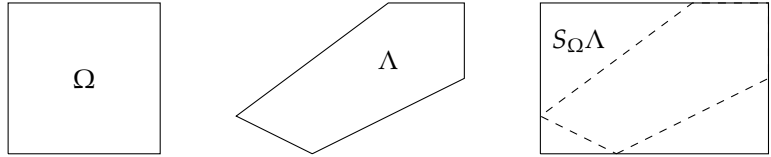

We call the resulting polygon $S_\Omega\Lambda$ the $\Omega$\textit{-stretching} of $\Lambda$. The following lemma is due to Michael Hutchings.

\begin{lemma} \label{lem:stretch} $\ell_\Omega(\partial\Lambda)=\ell_\Omega(\partial S_\Omega\Lambda)$.
\end{lemma}

\begin{proof} Let $p_1,\dots,p_k$ denote the vertices of $\Omega$. Let $q_i$ be a point on $\partial\Lambda$ such that a tangent vector to $\partial\Lambda$ at $q_i$ is parallel to the vector $p_i-p_{i-1}$. Then by definition, the $\Omega$-length of $\partial\Lambda$ is
$$\sum_ip_i\times(q_{i+1} - q_i)$$
Notice that the same points $q_i$ still satisfy the requirements for computing the $\Omega$-length of $\partial S_\Omega\Lambda$, so that the nothing changes in the expression of $\ell_\Omega(\partial S_\Omega\Lambda)$ from that for $\ell_\Omega(\partial\Lambda)$.
\end{proof}

The effect of $\Omega$-stretching is to produce a polygon of the same $\Omega$-length but with edges parallel to the edges of $\Omega$.

\subsection{Slope polytopes}

Let $\Omega$ be a rational convex domain. Denote its set of edges by $\on{Edge}(\Omega)$. An \textit{edge-orientation} $\mfk{o}$ of $\Delta$ is an orientation of each of its edges in such a way that the boundary of $\Delta$ is an oriented cycle. A polygon with an edge-orientation is called \textit{edge-oriented}. Given an edge-oriented convex lattice domain $\Omega$, define the \textit{slope} $v_e$ of an edge $e\in\on{Edge}(\Omega)$ to be the primitive lattice vector in the direction of the oriented edge. That is, of $e$ has endpoints $e_-$ and $e_+$ with orientation making $e_-$ the tail and $e_+$ the head, $v_e$ is the primitive ray generator of the ray $\R_{\geq0}\cdot(e_+-e_-)$.

\begin{definition} The \textit{slope polytope} of an edge-oriented convex lattice domain $\Omega$ is the lattice polytope
$$\on{Sl}(\Omega):=\on{Conv}(v_e:e\in\on{Edge}(\Omega))$$
\end{definition}

This produces a compact toric variety $Y_{\on{Sl}(\Omega)}$ on which the algebraic geometry side of the story will take place. We will actually work with a blowup of this toric variety, which we will denote by $\wt{Y}_{\on{Sl}(\Omega)}$.
\\

This blowup is obtained by creating a new fan by inserting rays through any slopes $v_e$ that are not vertices of $\on{Sl}(\Omega)$. For example, suppose that $\Omega$ has slopes $-e_1,e_2,e_1,e_1-e_2,e_1-2e_2$. The slope polytope only has vertices $-e_1,e_2,e_1,e_1-2e_2$ and so one extra ray has to be added for $e_1-e_2$. This is demonstrated pictorally below.
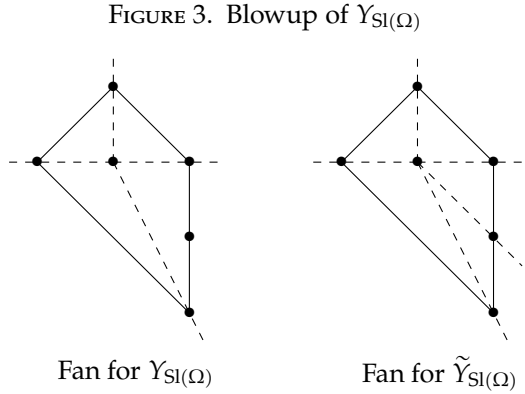
\begin{figure}[h]
\caption{Blowup of $Y_{\on{Sl}(\Omega)}$}
\begin{center}
\begin{tikzpicture}
\node (o) at (0,0){$\bullet$};
\node (a) at (-1,0){$\bullet$};
\node (b) at (0,1){$\bullet$};
\node (c) at (1,0){$\bullet$};
\node (d) at (1,-1){$\bullet$};
\node (e) at (1,-2){$\bullet$};
\node (al) at (-1.5,0){};
\node (bl) at (0,1.5){};
\node (cl) at (1.5,0){};
\node (el) at (1.25,-2.5){};

\draw (a.center) to (b.center) to (c.center) to (d.center) to (e.center) to (a.center);
\draw[dashed] (o.center) to (al);
\draw[dashed] (o.center) to (bl);
\draw[dashed] (o.center) to (cl);
\draw[dashed] (o.center) to (el);

\node (o) at (4,0){$\bullet$};
\node (a) at (3,0){$\bullet$};
\node (al) at (2.5,0){};
\node (b) at (4,1){$\bullet$};
\node (bl) at (4,1.5){};
\node (c) at (5,0){$\bullet$};
\node (cl) at (5.5,0){};
\node (d) at (5,-1){$\bullet$};
\node (dl) at (5.5,-1.5){};
\node (e) at (5,-2){$\bullet$};
\node (el) at (5.25,-2.5){};

\draw (a.center) to (b.center) to (c.center) to (d.center) to (e.center) to (a.center);
\draw[dashed] (o.center) to (al);
\draw[dashed] (o.center) to (bl);
\draw[dashed] (o.center) to (cl);
\draw[dashed] (o.center) to (dl);
\draw[dashed] (o.center) to (el);

\node (l1) at (0.3,-2.8){Fan for $Y_{\on{Sl}(\Omega)}$};
\node (l1) at (4.3,-2.8){Fan for $\wt{Y}_{\on{Sl}(\Omega)}$};
\end{tikzpicture}
\end{center}
\end{figure}
\\

We will denote the resulting fan for the blowup by $\wt{\Sigma}_{\on{Sl}(\Omega)}$. Observe that this fan is in some sense a rotation of the inner normal fan of $\Omega$ after picking bases, though they naturally live in dual lattices. At the end of this section we will provide an alternative version of the content below phrased in terms of the inner normal fan instead of the slope polytope. It can be favourable to use each of these perspectives at different times.

\subsection{Balance divisors}

As above, let $\Omega$ be a rational convex domain oriented clockwise with slope polytope $\on{Sl}(\Omega)$. We will subsequently always assume that $\Omega$ has this orientation. Define the $\Omega$\textit{-length} of a vector $v\in\R^2$ to be
$$\ell_\Omega(v):=v\times p_v$$
where $p_v$ is a boundary point of $\Omega$ such that the halfplane $p_v+\{u\in\R^2:u\times v\geq0\}$ contains $\Omega$. Recall that the two-dimensional cross product $u\times v$ of two vectors $u$ and $v$ is defined to be the determinant of the matrix with $u$ and $v$ as first and second columns respectively.

\begin{lemma} Suppose $\Omega$ is lattice (resp. rational). The $\Omega$-length is an integral (resp. rational) support function for the fan $\wt{\Sigma}_{\on{Sl}(\Omega)}$.
\end{lemma}

\begin{proof} Suppose $v,v'$ are adjacent slopes in $\Omega$. The $\Omega$-length applied to any vector $w\in\on{Cone}(v,v')=\sigma$ is given by
$$\ell_\Omega(w)=p\times w$$
where $p$ is the vertex shared between the two edges of slopes $v$ and $v'$ respectively. This is linear on the cone $\sigma$, which features in $\Sigma_{\on{Sl}(\Omega)}$ by definition and describes all full-dimensional cones in $\Sigma_{\on{Sl}(\Omega)}$ as $v,v'$ range over adjacent slopes. $\ell_\Omega$ is clearly integral on integral vectors when the vertices of $\Omega$ are lattice points, and similarly for the rational case.
\end{proof}

\begin{definition} \label{def:bal} The \textit{balance divisor} for $\Omega$ is the $\Q$-Cartier divisor $D_\Omega$ associated with the support function $-\ell_\Omega$. Notice the change in sign.
\end{definition}

\begin{cor} The coefficients of $D_\Omega$ as a Weil divisor are
$$a_v=\ell_\Omega(v)$$
for a (primitive) slope vector $v$ of $\Omega$.
\end{cor}

\begin{cor} $D_\Omega$ is ample.
\end{cor}

\begin{proof} It is a straightforward check that $-\ell_\Omega$ is a strictly convex function, which corresponds to $D_\Omega$ being ample.
\end{proof}

\begin{lemma} The polytope for $D_\Omega$ is the result of rotating $\Omega$ $90^\circ$ anticlockwise around the origin.
\end{lemma}

\begin{proof} 
Denote by $\Omega^\perp$ the rotated version of $\Omega$. The edges of $\Omega$ are by construction orthogonal to the rays of $\wt{\Sigma}_{\on{Sl}(\Omega)}$ and so there is a facet presentation of $\Omega^\perp$ coming from this fan or, equivalently, a divisor $D$ on $\wt{Y}_{\on{Sl}(\Omega)}$. Order the slopes $v_1,\dots,v_s$ with corresponding toric boundary divisors $D_1,\dots,D_s$. It suffices that the coefficient $a_i$ of $D$ along $D_i$ is the same as the corresponding coefficient in $D_\Omega$. We will now compute this directly. The edge $e_i$ of $P^\perp$ with slope $v_i$ is carved out by the orthogonal hyperplanes to $v_{i-1},v_i,v_{i+1}$. Suppose that $v_{i-1},v_{i+1}$ form a $\Z$-basis for $\Z^2$. They are independent over $\Q$ and the case when they are not a $\Z$-basis is similar. By a change of coordinates, suppose $v_{i-1}=(1,0),v_{i+1}=(0,-1)$ and $v_i=(\alpha,\beta)$. Then the endpoints of the edge in $\Omega^\perp$ corresponding to $v_i$ are
$$\left(-a_{i-1},\frac{\alpha a_{i-1}-a_i}{\beta}\right)\text{ and }\left(-\frac{\beta a_{i+1}+a_i}{\alpha},a_{i+1}\right)$$
After rotating back, the $\Omega$-length of $v_i$ is then
$$\ell_\Omega(v_i)=\left|\begin{array}{cc}
\alpha & \beta \\
a_{i+1} & \frac{\beta a_{i+1}+a_i}{\alpha}
\end{array}\right|=a_i$$
which is the same as the corresponding coefficient in $D_\Omega$.
\end{proof}

\begin{cor} $L_{r\Omega}=h^0(rD_\Omega)$.
\end{cor}

Notice that there are many choices of $\Omega$ with the same slope polytope $\on{Sl}(\Omega)$ and so to reflect the choice of $\Omega$ an extra choice has to be made in the geometry. This choice is a polarisation, where $\wt{Y}_{\on{Sl}(\Omega)}$ is polarised by the ample divisor $D_\Omega$. The same proof actually shows:

\begin{cor} \label{cor:lco} Let $\Lambda$ be a polygon with all edges parallel to edges of $\Omega$. Denote by $\Lambda^\perp$ the $90^\circ$ anticlockwise rotation of $\Lambda$ about the origin. The coefficients of a divisor $D_\Lambda$ on $\wt{Y}_{\on{Sl}(\Omega)}$ with polygon $\Lambda^\perp$ are
$$D_\Lambda=\sum\ell_\Lambda(v)D_v$$
with notation as above.
\end{cor}

When $\Omega$ is lattice, $D_\Omega$ is Cartier. Cartier divisors can also be characterised by their Cartier data, which has a toric version found in \S4.2 of \cite{cls}. To this end, let $(a,b)^\perp:=(-b,a)$. This has the property that $-u\cdot v^\perp=u\times v$.

\begin{cor} The Cartier data for $D_\Lambda$ is $m_{\sigma_i}=q_i^\perp$, where $q_i$ is the vertex in common between the edges of slopes $v_i,v_{i+1}$, the vertices in $\Sigma(\Omega)$ bounding $\sigma_i$.
\end{cor}

\begin{proof} As seen, $v_i\times q_i=a_i$ and so $v_i\cdot q_i^\perp=-a_i$.
\end{proof}

The balance divisor also captures the $\Omega$-length by how it intersects other divisors. We will prove the following lemma in toric geometry to progress towards this.

\begin{lemma} \label{lem:tech} Let $X_\Sigma$ be a projective toric surface. An $\R$-divisor $D$ on $X_\Sigma$ is nef iff $D\cdot D_\rho$ equals the lattice length of the edge of $P(D)$ corresponding to $\rho$ for each ray $\rho\in\Sigma(1)$.
\end{lemma}

\begin{proof} The if part is clear by the toric Kleiman condition. For the converse, observe that if $D$ is ample then there is a unique facet presentation of $\Lambda^\perp:=P(D)$ as every slope is represented by an edge in $P(D)$. This means that $D$ must be equal to
$$\sum\ell_\Lambda(u_\rho)D_\rho$$
adapting notation from Corollary \ref{cor:lco} and the result follows from the proof of that corollary. If $D$ is nef, then it must be the case that some of the inequalities in the facet presentation are only just redundant: that is, none of the hyperplanes have empty intersection with $P(D)$, but some might only intersect at a vertex. This follows as the interior of the nef cone is the ample cone, or from the description of nef and ample divisors in \cite{bat} Theorem 2.15 or \cite{cls} Theorem 6.4.9. It suffices to show that $D\cdot D_\rho=0$ for any $\rho$ giving a redundant hyperplane (that is, an edge of length $0$) but this follows from a direct calculation using \cite{cls} Prop. 6.4.4.
\end{proof}

Suppose that $\Lambda$ is a polygon with edges parallel to the edges of $\Omega$. As discussed above, there is a facet presentation of $\Lambda^\perp$ and so there is a nef divisor $D_\Lambda$ on $\wt{Y}_{\on{Sl}(\Omega)}$ with this as its polygon.

\begin{lemma} \label{lem:dot} $\ell_\Omega(\partial\Lambda)=D_\Lambda\cdot D_\Omega$.
\end{lemma}

\begin{proof} From Lemma \ref{lem:tech}, the lattice length of the edge of slope $v_i$ in $\Lambda$ is $D_\Lambda\cdot D_i$. The $\Omega$-length of the edge is thus $(D_\Omega\cdot D_i)\cdot\ell_\Omega(v_i)$. Summing all these up gives the $\Omega$-perimeter as
$$\ell_\Omega(\partial\Lambda)=\sum(D_\Lambda\cdot D_i)\cdot\ell_\Omega(v_i)=D_\Lambda\cdot\sum\ell_\Omega(v_i)D_i=D_\Lambda\cdot D_\Omega$$
as required.
\end{proof}

\begin{cor} $\ell_\Omega(\partial\Lambda)=\ell_\Lambda(\partial\Omega)$.
\end{cor}

\subsection{Proof of Theorem \ref{thm:main3}} \label{sec:pf3}

We are now in a position to convert the definition of ECH capacities and cap functions into purely algebro-geometric language.

\begin{thm} \label{thm:transl} Suppose $\Omega$ is a rational convex domain. Then
\begin{align*}
&c_k(X_\Omega)=\underset{D}{\on{min}}\{D\cdot D_\Omega:h^0(\wt{Y}_{\on{Sl}(\Omega)},D)\geq k+1\} \\
&\on{cap}_{X_\Omega}(r)=\underset{D}{\on{max}}\{h^0(\wt{Y}_{\on{Sl}(\Omega)},D):D\cdot D_\Omega\leq r\}
\end{align*}
where both extrema range over all nef $\Q$- or $\R$-divisors $D$ on $\wt{Y}_{\on{Sl}(\Omega)}$.
\end{thm}

\begin{proof} Since intersection with $D_\Omega$ describes the $\Omega$-length and the number of lattice points enclosed equals $h^0$, the only thing to check is that the extrema ranging over nef $\Q$- or $\R$-divisors is equivalent to ranging over convex lattice paths. We will focus on the real case from which it will be clear why the minima are achieved by rational nef divisors. We use nef divisors to ensure that each `edge length' $D\cdot D_i$ is nonnegative. Note that the two equalities in the theorem are equivalent and so we will focus only on the first. For convenience denote
$$c_k^\text{alg}(\wt{Y}_{\on{Sl}(\Omega)})=\on{inf}\{D\cdot D_\Omega:h^0(\wt{Y}_{\on{Sl}(\Omega)},D)\geq k+1\}$$
Note that a minimum really is attained. Indeed, pick a nef $\R$-divisor $D_\star$ with at least $k+1$ global sections. Then $c_k^\text{alg}(\wt{Y}_{\on{Sl}(\Omega)})\leq D_\star\cdot D_\Omega$ and the infimum is the same if we take it over all nef $\R$-divisors with $h^0(\wt{Y}_{\on{Sl}(\Omega)},D)\geq k+1$ and $D\cdot D_\Omega\leq D_\star\cdot D_\Omega$. Observe that this extra condition places an upper bound on each of the (nonnegative) lattice lengths of edges of the polygon $P(D)$ for such $D$. This infimum thus takes place over a compact region inside the (closed) nef cone and is therefore realised by some divisor.
\\

Suppose that $D=D_\Lambda$ realises this minimum. Its (rotated) polygon $\Lambda$ must have a lattice point on every edge as otherwise one could perturb the coefficient in the facet presentation for an edge with no lattice point to obtain a divisor with the same number of global sections but smaller intersection with $D_\Omega$. Notice that this implies that $D$ is a $\Q$-divisor. Let $\Lambda'$ be the convex hull of all lattice points in $\Lambda$. Note that $k'+1= L_\Lambda=L_{\Lambda'}$ for some $k'\geq k$. Then $S_\Omega\Lambda'=\Lambda$ by construction (as we assumed that $\Lambda$ has a lattice point on each edge) and so by Lemma \ref{lem:stretch} and Lemma \ref{lem:dot} we have $\ell_\Omega(\partial\Lambda')=\ell_\Omega(\partial S_\Omega\Lambda)=\ell_\Omega(\partial\Lambda)=D\cdot D_\Omega$. Now we will show that, potentially after translation, $\partial\Lambda'$ is a convex lattice path in the sense of Definition \ref{def:path}.
\\

$\Lambda$ has two distinguished (possible length $0$) edges of slopes $-e_1$ and $e_2$ by construction of $\on{Sl}(\Omega)$ that meet at a point $p_0$. For these edges to each contain a lattice point, they must each be subsets of affine lines of the form $(y=\beta)$ and $(x=\alpha)$ respectively for some $\alpha,\beta\in\Z$. Hence $p_0=(\alpha,\beta)\in\Z^2$ is a lattice point. We can thus use this lattice point to translate $\Lambda$ back to the origin without changing the pairing with $D_\Omega$ (the $\Omega$-length) or the dimension of global sections. By convexity $\Lambda'$ thus also contains two adjacent edges with slopes $-e_1$ and $e_2$. Since $\Lambda$ has slopes parallel to the slopes of $\Omega$ and is convex, the boundary of $\Lambda$ forms a convex rational path in the sense of Definition \ref{def:path}. It follows that the boundary of $\Lambda'$ forms a convex lattice path and hence features in the minimum of Theorem \ref{thm:opt} giving the combinatorial formula for $c_{k'}(X_\Omega)$. Consequently,
$$c_k(X_\Omega)\leq c_{k'}(X_\Omega)\leq\ell_\Omega(\partial\Lambda')=\ell_\Omega(\partial\Lambda)=D\cdot D_\Omega=c_k^\text{alg}(\wt{Y}_{\on{Sl}(\Omega)})$$
For the converse inequality, suppose that $\Lambda$ is a lattice polygon whose boundary $\partial\Lambda$ is a convex lattice path realising the minimum of Theorem \ref{thm:opt}. That is, $c_k(X_\Omega)=\ell_\Omega(\partial\Lambda)$ and $L_\Lambda=k+1$. Then $\Xi=S_\Omega\Lambda$ is a rational polygon with edges parallel to the edges of $\Omega$, which hence defines a nef $\Q$-divisor $D_\Xi$ on $\wt{Y}_{\on{Sl}(\Omega)}$. Now, using Lemma \ref{lem:stretch} and Lemma \ref{lem:dot},
$$c_k(X_\Omega)=\ell_\Omega(\partial\Lambda)=\ell_\Omega(\partial S_\Omega\Lambda)=D_\Xi\cdot D_\Omega$$
Notice that $S_\Omega\Lambda$ contains at least as many lattice points as $\Lambda$ and so $h^0(\wt{Y}_{\on{Sl}(\Omega)},D_\Xi)\geq k+1$ giving
$$c_k^\text{alg}(\wt{Y}_{\on{Sl}(\Omega)})\leq D_\Xi\cdot D_\Omega=c_k(X_\Omega)$$
which supplies the converse inequality.
\end{proof}

Notice that $c_k^\text{alg}(\wt{Y}_{\on{Sl}(\Omega)})$ uses $h^0\geq k+1$ instead of equality (as in the original optimisation problem for ECH capacities in Theorem \ref{thm:opt}) because there might not be divisors on $\wt{Y}_{\on{Sl}(\Omega)}$ with $k+1$ sections; for example, there are no divisors $D$ on $\pr^2$ with $h^0(\pr^2,D)=2$. Combinatorially, this comes from the fact that the lattice paths in Definition \ref{def:path} are allowed any rational slopes whereas the paths coming from divisors in Theorem \ref{thm:transl} must have edges parallel to edges of $\Omega$.

\subsection{A speculative digression}

Observe that one can try to define for any pair of a projective surface $Y$ and an ample divisor $A$ on $Y$
$$c_k^\text{alg}(Y,A):=\underset{\on{Nef}(Y)_\R}{\on{inf}}\{D\cdot A:h^0(Y,D)\geq k+1\}$$
taking the infimum again over the nef cone. It would be interesting to explore whether some of these sequences interact with symplectic capacities for other kinds of symplectic $4$-manifold, or to study the structure of their associated cap functions. We speculate that these cap functions are eventually quasipolynomial when $Y$ is an orbifold.

\subsection{Reformulation in terms of the inner normal fan}

There is another fan one can associate to a polytope $P$ now living in $M_\R$ called the \textit{inner normal fan} $\Sigma(P)$, which consists of cones in $N_\R$. For a polygon $P\subset\R^2$, this is the fan with rays generated by inward-pointing normals to each of the faces and with all two-dimensional cones between them included. Observe that, after picking a basis as we implicitly did above, the fan $\wt{\Sigma}_{\on{Sl}(\Omega)}$ for the blowup $\wt{Y}_{\on{Sl}(\Omega)}$ of $Y_{\on{Sl}(\Omega)}$ is the $90^\circ$ anticlockwise rotation of $\Sigma(\Omega)$: taking slopes is dual to taking normals.
\\

Completely analogously, we obtain a toric variety $Y_{\Sigma(Q)}$ that is isomorphic to the previous toric variety $\wt{Y}_{\on{Sl}(\Omega)}$ with an ample divisor $D_\Omega$ whose coefficient along the divisor $D_\rho$ is $\ell_\Omega(v)$, where $\rho$ is the ray generated by a normal to the edge of slope $v$. $(Y_{\Sigma(\Omega)},D_\Omega)$ has the same intersection theoretic and cohomological properties as the pair $(\wt{Y}_{\on{Sl}(\Omega)},D_\Omega)$ and so the results of the previous subsections exactly cross over to this setting.

\begin{thm} \label{thm:reform}
Suppose $X_\Omega$ is a rational convex toric domain. Then
\begin{align*}
c_k(X_\Omega)=\on{min}\{D\cdot D_\Omega:h^0(Y_{\Sigma(\Omega)},D)\geq k+1\} \\
\on{cap}_{X_\Omega}(r)=\on{max}\{h^0(Y_{\Sigma(\Omega)},D):D\cdot D_\Omega\leq r\}
\end{align*}
where both extrema range over all nef $\Q$- or $\R$-divisors on $Y_{\Sigma(\Omega)}$.
\end{thm}

We remark that the advantage of the inner normal fan in this context is its familiarity as a standard object of toric algebraic geometry, however the approach via slope polytopes is quite pleasing and may have better duality properties if a general explanation for this phenomenology via mirror symmetry exists. For the sake of familiarity and consistency with the introduction, we will continue to use $\Sigma(\Omega)$ instead of $\wt{\Sigma}_{\on{Sl}(\Omega)}$ for the remainder of the paper.

\subsection{Free convex toric domains} \label{sec:free}

One can also consider the situation when $\Omega\subset\R^2$ is a convex body that doesn't intersect the coordinate axes, which is where fibres of the moment map decrease in dimension and pick up nontrivial isotropy. We call such $X_\Omega$ \textit{free convex toric domains}. This was one of the situations originally considered in \cite{h11}. There is an analogous theorem there to Theorem \ref{thm:opt}. To state it, we define for such $\Omega$ a new pseudonorm $\ell_\Omega^{v_\star}$ depending on a vector $v_\star\in\Omega^\circ$ as follows. Consider $\Omega'=\Omega-v_\star$. This is now a polygon with the origin in its interior. We consider the norm $||\cdot||_{\Omega'}$ whose unit ball is $\Omega'$ and its dual norm on $(\R^2)^*$
$$||\phi||^*_{\Omega'}:=\on{max}\{\phi(v):v\in\Omega'\}$$
We identify $(\R^2)^*$ with $\R^2$ via the dot product, giving
$$||u||_{\Omega'}^*:=\on{max}\{u\cdot v:v\in\Omega'\}$$
Define the length in this pseudonorm of a polygonal path $\psi$ consisting of line segments $v_1,\dots,v_r$ to be
$$\ell_{\Omega}^{v_\star}(\psi):=\sum_{i=1}^r||v_i||_{\Omega'}^*$$

\begin{lemma}[\cite{h14}, Exercise 4.13] The length of closed polygonal paths measured in $\ell_\Omega^{v_\star}$ is independent of $v_\star$.
\end{lemma}

We denote the restriction of $\ell_\Omega^{v_\star}$ to closed polygonal paths by $\ell_\Omega'$ to indicate its independence of $v_\star$.

\begin{thm}[\cite{h11}, Theorem 1.11] \label{thm:free_hutch} Suppose $
\Omega\subset\R^2$ is a polygon that does not intersect either coordinate axis so that $X_\Omega$ is a free convex toric domain. Then
$$c_k(X_\Omega)=\on{min}\{\ell_\Omega'(\partial\Lambda):L_\Lambda=k+1\}$$
where the minimum ranges over lattice polygons $\Lambda$.
\end{thm}

As discussed in \cite{h14} Exercise 4.16 it is equivalent to take the minimum over all polygons with edges parallel to edges of $\Omega$ and with no constraints on their vertices with the modification that $L_\Lambda\geq k+1$.

\begin{thm} \label{thm:free}
Suppose $X_\Omega$ is a free rational convex toric domain. Then,
\begin{align*}
c_k(X_\Omega)=\on{min}\{D\cdot D_\Omega:h^0(Y_{\Sigma(\Omega)},D)\geq k+1\} \\
\on{cap}_{X_\Omega}(r)=\on{max}\{h^0(Y_{\Sigma(\Omega)},D):D\cdot D_\Omega\leq r\}
\end{align*}
where $D_\Omega$ is the balance divisor from Definition \ref{def:bal} and where both extrema range over all nef $\Q$- or $\R$-divisors on $Y_{\Sigma(\Omega)}$.
\end{thm}

\begin{proof} As before, the two equalities are equivalent and so we will only show the first. Let $v_\star\in\Omega^\circ$ and set $\Omega'=\Omega-v_\star$. Suppose that $u_1,u_2$ are outward normals to adjacent faces of $\Omega$. The dual norm $||\cdot||_{\Omega'}^*$ is linear on $\on{Cone}(u_1,u_2)$, since the maximum of $v\cdot-$ will be achieved (possibly non-uniquely) at the vertex shared between the two adjacent edges for any $v\in\on{Cone}(u_1,u_2)$. It is hence a support function on the outer normal fan $\Sigma^-(\Omega)$, which is just the negative of the inner normal fan. Notice that for $v\in\on{Cone}(u_1,u_2)$, the dual norm $||v||_{\Omega'}^*=v\cdot p$ where $p$ is the vertex described above, but this is equal to $-v^\perp\times p$ by definition. Note that $p$ is exactly the point of $\partial\Omega'$ at which $-v^\perp$ is tangent to $\partial\Omega'$ so that $p=p_v$ as in the definition of $\Omega'$-length in \S\ref{sec:comb}. Hence,
$$||v||_{\Omega'}^*=\ell_{\Omega'}(-v^\perp)$$
It follows that
\begin{align*}
c_k(X_\Omega)&=\on{min}\{\ell'_\Omega(\partial\Lambda):L_\Lambda=k+1\} \\
&=\on{min}\{\ell_{\Omega'}(\partial\Xi):L_\Xi=k+1\}
\end{align*}
via the correspondence $\Lambda\mapsto-\Lambda^\perp$, where both minima range over all lattice polygons $\Lambda$ or $\Xi$ respectively. But by a similar (actually simpler) argument to the proof of Theorem \ref{thm:transl}, this second minimum can be seen to be equal to $\on{min}\{D\cdot D_{\Omega'}:h^0(Y_{\Sigma(\Omega)},D)\geq k+1\}$. Now $\Omega'$ is just a translate of $\Omega$ and so $D\cdot D_\Omega=D\cdot D_{\Omega'}$ for all divisors $D$, which gives the result.
\end{proof}

We finally observe that all of the machinery developed above works equally well when $\Omega$ is an irrational polygon with rational slopes, since rationality is only required on the level of edges to define a fan that will produce a toric variety. The only difference is that $D_\Omega$ will no longer be a $\Q$-divisor.

\section{Computing cap functions}

The aim of this section is to define `tightly constrained' convex domain and to prove the following theorem.

\begin{thm} \label{thm:cap} Suppose $\Omega$ is a tightly constrained convex lattice domain. Then, there exists $x_0\in\Z_{\geq0}$ such that for all $x\geq x_0$ and for each $r=0,\dots,\lambda-1$, 
$$\on{cap}_{X_\Omega}(r+\lambda x)=\on{ehr}_\Omega(x)+rx+\gamma_r$$
for some constant $\gamma_r\in\Z$ depending only on $r$. If $\Omega$ has a weight equal to $1$ then $\Omega$ is tightly constrained and moreover one can choose $x_0=0$.
\end{thm}

In order to do so, we will study the combinatorics of $\Omega$ in terms of its weight sequence, and then use this data to compute the cap function recursively. We will discuss the tightly constrained assumption on $\Omega$ and how every convex toric lattice domain conjecturally reduces to this case.

\subsection{Combinatorics of weight sequences}

Recall that the weight sequence associated to a convex domain $\Omega$ consists of a number and two lists that we will write as $(c;a_i;b_i)$. We will assume that the lists are finite sets of integers, which implies that $\Omega$ is rational. From the asymptotics of capacities of convex domains,
$$\on{Vol}(\Omega_2)=\on{Vol}(\amalg_i B(a_i))=\frac{1}{2}\sum a_i^2$$
and so
$$\ell_\Omega(\partial\Omega)=2\on{Vol}(\Omega)=\on{Vol}(B(c))-\on{Vol}(\amalg_i B(a_i))-\on{Vol}(\amalg_i B(b_i))=c^2-\sum a_i^2-\sum b_i^2$$
Consider now the number of lattice points enclosed by a concave domain, excluding those on the upper boundary. Each ball $B(b_i)$ contributes $\frac{1}{2}b_i(b_i+1)$ lattice points; note that the transformation realising the inductive description of the weight sequence is a special affine linear map and so preserves lattice point counts. Hence the number of lower lattice points (i.e. excluding the upper boundary) in $\Omega_3$ is
$$\sum\frac{1}{2}b_i(b_i+1)$$
and thus the number of lattice points enclosed by $\Omega$ is
$$\frac{1}{2}(c+1)(c+2)-\sum\frac{1}{2}\alpha_i(\alpha_i+1)-\sum\frac{1}{2}b_j(b_j+1)$$
For future reference will note that this is equal to
$$1+\frac{1}{2}c(c+3)-\sum\frac{1}{2}\alpha_i(\alpha_i+1)-\sum\frac{1}{2}b_j(b_j+1)$$

\subsection{Reducing the problem}

For a convex domain $\Omega$ with weight sequence $w(\Omega)=(c;a_i;b_j)$, Lemma \ref{lem:sharp} gives that the ECH capacities of $X_\Omega$ are given by
$$c_k(X_\Omega)=\on{min}\{c_{k+k_2+k_3}(B(c))-c_{k_2}(\amalg_i B(a_i))-c_{k_3}(\amalg_j B(b_j)):k_2,k_3\in\Z_{\geq0}\}$$
By the disjoint union property of capacities, this is equal to
$$c_k(X_\Omega)=\on{min}\{c_{k+\sum_i k_i+\sum_j m_j}(B(c))-\sum_i c_{k_i}(B(a_i))-\sum_j c_{m_j}(B(b_j)):k_i,m_j\in\Z_{\geq0}\}$$
It follows that the cap function of $X_\Omega$ is given by
$$\on{cap}_{X_\Omega}(r)=1+\on{max}\{k:\exists k_i,m_j\text{ with }c_{k+\sum_i k_i+\sum_j m_j}(B(c))-\sum_i c_{k_i}(B(a_i))-\sum_j c_{m_j}(B(b_j))\leq r\}$$
The capacities of a ball $B(q)$ take the form
$$c_k(B(q))=dq\text{ when $\frac{1}{2}d(d+1)\leq k\leq\frac{1}{2}\delta(\delta+3)$}$$
Hence, to maximise $k$, one may assume that $k_i=\frac{1}{2}\alpha_i(\alpha_i+1)$, $m_j=\frac{1}{2}\beta_j(\beta_j+1)$, and $k+\sum_i k_i+\sum_j m_j=\frac{1}{2}\delta(\delta+3)$ for some $\alpha_i,\beta_j,\delta$. Therefore $\on{cap}_{X_\Omega}(r)$ is $1$ plus the maximum of
$$C(\delta,\alpha_i,\beta_j):=\frac{1}{2}\delta(\delta+3)-\sum_i\frac{1}{2}\alpha_i(\alpha_i+1)-\sum_j\frac{1}{2}\beta_j(\beta_j+1)$$
subject to
$$\delta c-\sum_i \alpha_ia_i-\sum_j \beta_ib_i\leq r$$
where $\alpha_i,\beta_j,\delta$ range over nonnegative integers.

\subsection{Final calculations}

\begin{lemma} Fix a weight sequence $(c;a_1,\dots,a_s;b_1,\dots,b_t)$ and let $\lambda=c^2-\sum a_i^2-\sum b_i^2$. Suppose $(\delta,\alpha_i,\beta_i)$ maximises $C(\delta,\alpha_i,\beta_j)$ subject to
$$\delta c-\sum_i \alpha_ia_i-\sum_j \beta_jb_j=r$$
Then the sequence $(\delta+c,\alpha_i+a_i,\beta_i+b_i)$ maximises $C(\delta',\alpha_i',\beta_i')$ subject to
$$\delta'c-\sum_i \alpha_i'b_i-\sum_j \beta_j'b_j=r+\lambda$$
\end{lemma}

\begin{proof} Suppose there exists $(\delta',\alpha_i',\beta_j')$ with $C(\delta',\alpha_i',\beta_j')>C(\delta+c,\alpha+a_i,\beta_j+b_j)$. We will show that $C(\delta'-c,\alpha_i'-a_i,\beta_j'-b_j)>C(\delta,\alpha_i,\beta_j)$, contradicting maximality since
$$(\delta'-c)c-\sum(\alpha_i'-a_i)a_i-\sum(\beta_j'-b_j)b_j=r$$
For convenience, relabel the $b_j$ as $a_{s+j}$ and $\beta_j$ as $\alpha_{s+j}$ and write $C(\delta,\alpha_i)=C(\delta,\alpha_i,\beta_j)$.
Compute $2C(\delta'-c,\alpha_i'-a_i)$ to be
\begin{align*}
&(\delta'-c)(\delta'-c+3)-\sum (\alpha_i'-a_i)(\alpha_i'-a_i+1) \\
&=\delta'(\delta'+3)-\sum \alpha_i'(\alpha_i'+1)-c\delta'+\sum\alpha_i'a_i-c(\delta'+3)+\sum(\alpha_i'+1)a_i+c^2-\sum a_i^2 \\
&= \delta'(\delta'+3)-\sum \alpha_i'(\alpha_i'+1)-(r+\lambda)-(r+\lambda)-3c+\sum a_i+\lambda \\
&>(\delta+c)(\delta+c+3)-\sum (\alpha_i+a_i)(\alpha_i+a_i+1)-2r-\lambda-3c+\sum a_i \\
&=\delta(\delta+3)-\sum \alpha_i(\alpha_i+1)+c\delta-\sum \alpha_ia_i+c(\delta+3)-\sum(\alpha_i+1)a_i+c^2-\sum a_i^2-2r-\lambda-3c+\sum a_i \\
&=C(\delta,\alpha_i,\beta_j)+r+r+3c-\sum a_i+\lambda-2r-\lambda-3c+\sum a_i \\
&=C(\delta,\alpha_i,\beta_j)
\end{align*}
as desired.
\end{proof}

\begin{definition} \label{def:tight} Say that a convex lattice domain $\Omega$ (or a convex lattice toric domain $X_\Omega$) with weight sequence $(c;a_i;b_i)$ is \textit{tightly constrained with lower bound $r_0$} if for all $r\geq r_0$ the maximum of $C(\delta,\alpha_i,\beta_j)$ subject to 
$$\delta c-\sum_i \alpha_ia_i-\sum_j \beta_ib_i\leq r$$
is attained by some $(\delta,\alpha_i,\beta_j)$ with
$$\delta c-\sum_i \alpha_ia_i-\sum_j \beta_ib_i=r$$
Say that $\Omega$ is \textit{tightly constrained} if it is tightly constrained with some lower bound $r_0$.
\end{definition}

\begin{lemma} $\Omega$ being tightly constrained with lower bound $r_0$ is equivalent to the statement that for every positive integer $r\geq r_0$ there is some $k\in\Z_{\geq0}$ such that $c_k(X_\Omega)=r$.
\end{lemma}

\begin{proof} By definition the cap function of $X_\Omega$ is $1$ plus the largest value of $k$ such that $c_k(X_\Omega)\leq r$. If there is some $k$ with $c_k(X_\Omega)=r$ then this largest value of $k$ will be achieved by some $k$ with $c_k(X_\Omega)=r$ by monotonicity. The largest value of $k$ corresponds to a value of $C(\delta,\alpha_i,\beta_j)$ from the reasoning above, for which the corresponding capacity takes the value $\delta c-\sum\alpha_ia_i-\sum\beta_jb_j=r$.
\end{proof}

Equivalently, $\on{cap}_{X_\Omega}(r+1)>\on{cap}_{X_\Omega}(r)$ for all $r\geq r_0$, so that $\on{cap}_{X_\Omega}$ is eventually strictly increasing.

\begin{example} Suppose $X_\Omega=E(a,b)$ is an ellipsoid with $a$ prime, $a<b$, and $\on{gcd}(a,b)=1$. Then $X_\Omega$ is tightly constrained with lower bound $(a-1)b$ to cover all residues mod $a$.
\end{example}

\begin{lemma} \label{lem:wt} Suppose $\Omega$ has at least one weight equal to $1$. Then $\Omega$ is tightly constrained with lower bound $r_0=0$.
\end{lemma}

\begin{proof} Suppose $(c;a_i;b_j)$ is a weight sequence with $a_1=1$. Let $(\delta,\alpha_i,\beta_j)$ maximise $C(\delta,\alpha_i,\beta_j)$ subject to $\delta c-\sum_i\alpha_ia_i-\sum_j\beta_jb_j\leq r$. Suppose $\delta c-\sum_i\alpha_ia_i-\sum_j\beta_jb_j<r$. Modify $(\delta,\alpha_i,\beta_j)$ to $(\delta,\alpha_i',\beta_j)$ where $\alpha_1'=\alpha_1-1$ and $\alpha_i'=\alpha_i$ for $i\geq2$. This sequence has $C(\delta,\alpha_i',\beta_j)=\frac{1}{2}\delta(\delta+3)-\frac{1}{2}(\alpha_1-1)\alpha_1-\sum\frac{1}{2}\alpha_i(\alpha_i+1)-\sum \frac{1}{2}\beta_j(\beta_j+1)>C(\delta,\alpha_i,\beta_j)$ and $\delta c-(\alpha_1-1)a_1-\sum \alpha_ia_i-\sum\beta_jb_j=\delta c-\sum_i\alpha_ia_i-\sum_j\beta_jb_j+b_1\leq r$. This contradicts the fact that $(\delta,\alpha_i,\beta_j)$ was maximal.
\end{proof}

\begin{conjecture} \label{conj:wt} Suppose that $\on{gcd}\{c,a_1,\dots,a_s,b_1,\dots,b_t\}=1$. Then $\Omega$ is tightly constrained.
\end{conjecture}

Notice that the conjecture will certainly fail for weight sequences without the coprimality assumption. For example, the ball $B(2)$ has capacities that are all even numbers and so there can be no odd values of the constraint. We will henceforth make the assumption that $\Omega$ is tightly constrained.

\begin{cor} \label{cor:poly} For a tightly constrained convex lattice domain $\Omega$ with lower bound $r_0=0$, $\partial\Omega$ is an optimal path among lattice paths of length at most $\ell_\Omega(\partial\Omega)$.
\end{cor}

\begin{proof} Clearly $\on{cap}_{X_\Omega}(0)=1+0$ is attained by $(\delta,\alpha_i,\beta_j)=(0,0,\dots,0)$. Hence,
$$\on{cap}_{X_\Omega}(\lambda)=1+C(c,a_i,b_j)=\frac{1}{2}(c+1)(c+2)-\sum a_i(a_i+1)-\sum b_j(b_j+1)=L_\Omega$$
as required.
\end{proof}

\begin{lemma} Let $\Omega$ be a tightly constrained convex lattice domain. Denote the $\Omega$-perimeter of $\Omega$ by $\lambda$. Then there exists $x_0\in\Z_{\geq0}$ such that for all $x\geq x_0$ and for each $r=0,\dots,\lambda-1$,
$$\on{cap}_{X_\Omega}(\lambda x+r)=\on{Vol}(\Omega)x^2+(\frac{1}{2}L_{\partial\Omega}+r)x+\gamma_r$$
for some $\gamma_r\in\Z$, where $L_{\partial\Omega}$ is the number of lattice points on the boundary of $\Omega$.
\end{lemma}

\begin{proof} From Lemma \ref{lem:wt} and the assumption that $\Omega$ is tightly constrained one has that the maximum value of $C(\delta,\alpha_i,\beta_j)$ subject to $\delta c-\sum_i\alpha_ia_i-\sum_j\beta_jb_j\leq r+\lambda$ is
$$C(\delta',\alpha_i',\beta_j')+r+\frac{1}{2}c(c+3)-\sum\frac{1}{2}b_i(b_i+1)=C(\delta',\alpha_i',\beta_j')+r+L_\Omega-1$$
when $(\delta',\alpha_i',\beta_j')$ is maximal subject to $\delta' c-\sum_i\alpha_i'a_i-\sum_j\beta_j'b_j\leq r$, at least for large enough $r$. It follows that, for $r+\lambda x$ large enough,
\begin{equation} \tag{$*$}
\on{cap}_{X_\Omega}(r+\lambda(x+1))=\on{cap}_{X_\Omega}(r+\lambda x)+r+\lambda x+L_\Omega-1
\end{equation}
This implies that $\on{cap}_{X_\Omega}(r+\lambda x)$ is eventually a quadratic polynomial. Solving the difference equation ($*$) gives the leading term as $\lambda/2$ and gives the linear coefficient as $L_\Omega-\on{Vol}(\Omega)-1+r$. By Pick's formula the linear term is equal to $\frac{1}{2}L_{\partial\Omega}+r$, and we have seen that $\lambda/2=\ell_\Omega(\partial\Lambda)/2=\on{Vol}(\Omega)$.
\end{proof}

This is the desired quasipolynomial representation of $\on{cap}_{X_\Omega}$. However, we would also like this to have an algebro-geometric interpretation. The Ehrhart polynomial of $\Omega$, as a lattice polygon, is
$$\on{ehr}_\Omega(x)=\on{Vol}(\Omega)x^2+\frac{1}{2}L_{\partial\Omega}x+1$$

\begin{cor} \label{cor:main} Let $\Omega$ be a tightly constrained convex lattice domain of $\Omega$-perimeter $\lambda$. Then, for any $r\in\{0,1,\dots,\lambda-1\}$ and sufficiently large $x\in\Z_{\geq0}$
\begin{align*}
\on{cap}_{X_\Omega}(r+\lambda x)&=\on{ehr}_\Omega(x)+rx+\gamma_r \\
&=\on{hilb}_{(Y_{\Sigma(\Omega)},D_\Omega)}(x)+rx+\gamma_r
\end{align*}
for some $\gamma_r\in\Z$. In particular, for all sufficiently large $x\in\Z_{\geq0}$
$$\on{cap}_{X_\Omega}(\lambda x)=\on{ehr}_\Omega(x)+\gamma_0=\on{hilb}_{(Y_{\Sigma(\Omega)},D_\Omega)}(x)+\gamma_0$$
\end{cor}

We believe that always $\gamma_r=\on{cap}_{X_\Omega}(r)-1$, which is what one would obtain from the difference equation ($*$) holding for all $x\in\Z$, not just all sufficiently large $x$. This would in particular imply that $\gamma_0=0$. Suppose $X_\Omega$ is not tightly constrained. Assuming Conjecture \ref{conj:wt}, one can scale $\Omega$ to obtain a convex lattice domain $\Omega'$ that is tightly constrained. Let $q\Omega'=\Omega$. Then, using the scaling axiom from \S\ref{sec:prop}, for any $r=0,\dots,q-1$ one has
$$\on{cap}_{X_\Omega}(r+qx)=\on{cap}_{X_\Omega}(qx)=\on{cap}_{X_{\Omega'}}(x)$$
Thus, knowing Theorem \ref{thm:cap} for tightly constrained convex toric lattice domains is sufficient to completely describe the long term behaviour of the cap function for all convex toric lattice domains.

\begin{example} For $X_\Omega=B(2)$, one has
$$\on{cap}_{X_\Omega}(r)=
\begin{cases}
\on{cap}_{B(1)}(\frac{r}{2}) & r\equiv0\on{mod}{2} \\
\on{cap}_{B(1)}(\frac{r-1}{2}) & r\equiv1\on{mod}{2}
\end{cases}=\begin{cases}
\frac{1}{8}(r+2)(r+4) & r\equiv0\on{mod}{2} \\
\frac{1}{8}(r+1)(r+3) & r\equiv1\on{mod}{2}
\end{cases}$$
\end{example}

We conjecture that the word `eventually' may be dropped in all the above results, and that in fact the cap function of a tightly constrained convex toric lattice domain is given entirely by the quasipolynomial in Theorem \ref{thm:cap}. Of course, this is already proven if one of the weights of $\Omega$ is equal to $1$.


\begin{thebibliography}{1}
\bibitem[B91]{bat} Batyrev, V. V. (1991). \textit{On the classification of smooth projective toric varieties.} Tohoku Mathematical Journal, Second Series, 43(4), 569-585.
\bibitem[CCFHR]{ccfhr} Choi, K., Cristofaro-Gardiner, D., Frenkel, D., Hutchings, M., \& Ramos, V. G. B. (2014). \textit{Symplectic embeddings into four-dimensional concave toric domains.} Journal of Topology, 7(4), 1054-1076.
\bibitem[CHLS]{chls} Cieliebak, K., Hofer, H., Latschev, J., \& Schlenk, F. (2005). \textit{Quantitative symplectic geometry}. arXiv preprint math/0506191.
\bibitem[C15]{cg} Cristofaro-Gardiner, D. (2014). \textit{Symplectic embeddings from concave toric domains into convex ones.} arXiv preprint arXiv:1409.4378.
\bibitem[CHR]{chr} Cristofaro-Gardiner, D., Hutchings, M., \& Ramos, V. G. B. (2015). \textit{The asymptotics of ECH capacities.} Inventiones mathematicae, 199(1), 187-214.
\bibitem[CK13]{ck} Cristofaro-Gardiner, D., \& Kleinman, A. (2013). \textit{Ehrhart polynomials and symplectic embeddings of ellipsoids.} arXiv preprint arXiv:1307.5493.
\bibitem[CS18]{cs} Cristofaro-Gardiner, D., \& Savale, N. (2018). \textit{Sub-leading asymptotics of ECH capacities.} arXiv preprint arXiv:1811.00485.
\bibitem[CLS]{cls} Cox, D. A., Little, J. B., \& Schenck, H. K. (2011). \textit{Toric varieties.} American Mathematical Soc..
\bibitem[G85]{grom} Gromov, M. (1985). \textit{Pseudo holomorphic curves in symplectic manifolds.} Inventiones mathematicae, 82(2), 307-347.
\bibitem[HJST]{hjst} Haase, C., Juhnke-Kubitzke, M., Sanyal, R., \& Theobald, T. (2015). \textit{Mixed Ehrhart polynomials.} arXiv preprint arXiv:1509.02254.
\bibitem[H11]{h11} Hutchings, M. (2011). \textit{Quantitative embedded contact homology.} Journal of Differential Geometry, 88(2), 231-266.
\bibitem[H14]{h14} Hutchings, M. (2014). \textit{Lecture notes on embedded contact homology.} In Contact and symplectic topology (pp. 389-484). Springer, Cham.
\bibitem[M11]{mcd} McDuff, D. (2011). \textit{The Hofer conjecture on embedding symplectic ellipsoids.} Journal of Differential Geometry, 88(3), 519-532.
\end{thebibliography}
\end{document}